\documentclass[12pt]{article}
\usepackage[utf8]{inputenc}
\usepackage{geometry}
\usepackage{url}
\usepackage{todonotes}
\usepackage{amsmath,amssymb,amsthm}
\usepackage{subcaption}
\newtheorem{theorem}{Theorem}

\newtheorem{lemma}[theorem]{Lemma}
\newtheorem{corollary}[theorem]{Corollary}
\newtheorem{proposition}[theorem]{Proposition}
\newtheorem{observation}[theorem]{Observation}
\newtheorem{conjecture}{Conjecture}

\theoremstyle{definition}                    
\newtheorem{definition}[theorem]{Definition}
\theoremstyle{remark}

\title{Intersection sizes of linear subspaces with the hypercube}
\author{Carla Groenland and Tom Johnston\thanks{Mathematical Institute, University of Oxford, Oxford, OX2 6GG, United Kingdom.\newline  E-mail: \texttt{\{groenland,johnston\}@maths.ox.ac.uk}.}}
\date{\today}

\renewcommand{\H}{\mathbb{H}}

\newcommand{\N}{\mathbb{N}}
\newcommand{\R}{\mathbb{R}}

\usepackage{tikz}
\tikzstyle{vertex} = [fill,shape=circle, minimum size = 0.5cm, inner sep = 0, node distance = 1cm]

\begin{document}
\maketitle
\begin{abstract}
We continue the study by Melo and Winter [\emph{arXiv:1712.01763}, 2017] on the possible intersection sizes of a $k$-dimensional subspace with the vertices of the $n$-dimensional hypercube in Euclidean space. Melo and Winter conjectured that all intersection sizes larger than $2^{k-1}$ (the ``large'' sizes) are of the form $2^{k-1}+2^i$. We show that this is almost true: the large intersection sizes are either of this form or of the form $35\cdot 2^{k-6}$. We also disprove a second conjecture of Melo and Winter by proving that a positive fraction of the ``small'' values is missing.
\end{abstract}
\section{Introduction}
What possible intersection sizes can the hypercube $\H^n=\{0,1\}^n$ have with a $k$-dimensional linear subspace $S$ of $\R^n$? The smallest possible intersection size is 1 and the largest $2^k$, but which numbers in between are possible? Melo and Winter \cite{MeloWinter} initiated the study of the sets $H(n,k)$ of possible intersection sizes.
Melo and Winter made two conjectures about the structure of $H(\infty,k)=\bigcup_{n\geq k}H(n,k)$, depending on whether the intersection size is ``large'' (greater than $2^{k-1}$) or ``small'' (smaller than $2^{k-1}$). Let $H^+(n,k)$ be the large values of $H(n,k)$ i.e. \[H^+(n,k)  = H(n,k) \setminus \{1,\dots,2^{k-1}\}\] and set $H^+(\infty,k)= H(\infty, k ) \setminus \{1, \dots, 2^{k-1}\}$. Similarly, we denote the small values by $H^-(n,k)=H(n,k)\cap \left[2^{k-1}\right]$ and $H^-(\infty,k)=H(\infty,k)\cap \left[2^{k-1}\right]$ (where we use the standard notation $[n] = \{1,\dots, n\}$). Using this notation, Melo and Winter's conjectures are as follows.
\begin{conjecture}[Conjecture 3.2 in \cite{MeloWinter}]
\label{conj:large}
$H^+(\infty,k)=\{2^{k-1}+2^i:i\in \{0,1,\dots,k-1\}\}$.
\end{conjecture}
\begin{conjecture}[Conjecture 3.3 in \cite{MeloWinter}]
\label{conj:small}
$H^-(\infty,k)=[2^{k-1}]$.
\end{conjecture}
In Section \ref{sec:large} we determine the exact structure of $H^+(n,k)$ with the following theorem and find that there is one additional value in $H^+(n,k)$ for $k \geq 6$. 
\begin{theorem}
\label{thm:largeinf} For $k \geq 6$,
\[H^+(\infty,k)=\{2^{k-1}+2^i:i\in \{0,\dots,k-1\}\}\cup\{2^{k-1}+2^{k-5}+2^{k-6}\}\]
\end{theorem}
In fact, by proving a precise structural result (Lemma \ref{lem:structure}) and combining this with a computer search, we determine $H^+(n,k)$ for every $n$ (Theorem \ref{thm:largen}) from which the theorem above follows.  In particular, we are also able to answer another question raised by Melo and Winter \cite{MeloWinter} when restricted to the large intersections, namely the smallest $n$ such that $t\in H^+(n,k)$. 

In Section \ref{sec:small} we disprove Conjecture \ref{conj:small} with the following result.

\begin{theorem}
\label{thm:smalln}
For $k \geq 8$ and any $m \geq 1$,
\[ H(k+m, k) \cap \left\{ \tfrac{15}{16} 2^{k-1}, \dots, 2^{k-1}\right\}= \left\{ \tfrac{15}{16} 2^{k-1}, \tfrac{126}{128} 2^{k-1}, 2^{k-1}  \right\}. \]
\end{theorem}

\section{Framework and notation}
\label{sec:framework}
Melo and Winter showed that, after permuting coordinates as necessary, any linear subspace $S$ of dimension $k$ can be parameterised as $S=\{v\oplus L(v):v\in \R^k\}$ for some linear map $L:\R^k\to \R^{n-k}$, and so $S\cap \H^{n}=\{(v,L(v)):v\in \H^k,~ L(v)\in \H^{n-k}\}$. This led them to consider the equivalent formulation
\[
H(n, k) = \{t \mid \exists L : \R^k\to \R^{n-k} \text{ linear},~ t = |\H^k\cap L^{-1}\H^{n-k}|\}
\]
that we will also use. 
 
Throughout the text, $L:\R^k\to \R^m$ is a linear map and $m=n-k$. Let $e_i\in \R^k$ denote the vector with a 1 in position $i$ and 0s elsewhere, so that $\{e_1, \dots, e_k\}$ is the standard basis for $\mathbb{R}^k$. Any linear map is determined by the values it takes on these basis elements. We denote
\[
I(L)=L^{-1}\left(\H^m\right)\cap \H^k
\]
for the intersection pattern associated with $L$ and $t(L)=|I(L)|$ for its size. The \textbf{support} $S(L)=\{i\in \{1,\dots,k\}:L(e_i)\neq 0\}$ denotes the indices of the basis vectors on which $L$ is non-zero.
\begin{observation}
After permuting coordinates as necessary, there exists a $J(L)\subseteq \H^{|S(L)|}$ such that $I(L)=J(L)\times \H^{k-|S(L)|}$.
\end{observation}
For $A\subseteq \{1,\dots,m\}$, define the linear map $L_A:\R^k\to \R^{|A|}$ by $L_A = \pi_A \circ L$, where $\pi_A$ is the projection map onto the coordinates in the set $A$, for example $\pi_{\{1,3\}}(x_1, \dots, x_m) = (x_1, x_3)$. We use the shorthand notation $L_i:=L_{\{i\}}$.
Note that $L(x) \in \mathbb{H}^m$ if and only if $L_i(x) \in \{0,1\}$ for $i = 1, \dots, m$ and that adding conditions can only decrease the intersection size:
\[
I\left(L_{A\cup\{i\}}\right)=\bigcap_{j\in A\cup\{i\}}I(L_j)\subseteq \bigcap_{j\in A}I(L_j)=I(L_{A})
\] for each $A\subseteq [m]$ and $i\in[m]$.
\begin{definition}
\label{def:minimal}
A linear map $L:\R^k\to \R^m$ is called \textbf{minimal} if $S(L_i)\not\subseteq \bigcup_{j\neq i}S(L_j)$ and $t(L_i) < 2^k$ for all $i\in \{1,\dots,m\}$.
\end{definition}

We will see that only specific ``shapes'' of the sets $S(L_i)$ can ``cover'' $[m]$ and still give a large intersection value $t(L)$. 
\begin{definition}
\label{def:shape}
The \textbf{shape} of the linear map $L:\R^k\to\R^m$ is the hypergraph $(V,E)$ with vertex set $V=\{1,\dots,k\}$ and edge set $\{S(L_i):i\in [m]\}$.
\end{definition}
Two specific shapes will be of particular interest. 
\begin{itemize}
    \item A \textbf{(2,1)-star} is a 2-uniform hypergraph $(V,E)$ for which there is a centre $c\in V$ such that $E=\{\{c,v\}:v\in V\setminus \{c\}\}$.
    \item  A \textbf{(3,2)-star} is a 3-uniform hypergraph $(V,E)$ for which there exist distinct $c_1,c_2\in V$ such that $E=\{\{c_1,c_2,v\}:v\in V\setminus \{c_1,c_2\}\}$.
\end{itemize} We will write ``the sets $S(L_i)$ form a (3,2)-star'' to mean that the hypergraph formed by taking the sets $S(L_i)$ as the edge set (i.e. the shape of $L$) is a (3,2)-star.
\begin{figure}
    \centering
    \begin{subfigure}[t]{0.4\textwidth}
    \centering
    \resizebox{\textwidth}{!}{%
    \begin{tikzpicture}
    \node[vertex] (v1) at (0,0) {};
    \node[vertex] (v2) at (2,-1) {};
    \node[vertex] (v3) at (2,1) {};
    \node[vertex] (v4) at (4,0) {};
    \node[vertex] (v5) at (6,0) {};

    \draw[line width=0.05cm] plot [smooth cycle] coordinates {(-0.5,0) (2.35,1.35) (2.35, -1.35)};
    \draw[line width=0.05cm] plot [smooth cycle] coordinates {(4.5,0) (1.65,1.35) (1.65, -1.35)};
    \draw[line width = 0.05cm] plot [smooth cycle] coordinates{(3.5, 0) (4,0.5) (6, 0.5) (6.5, 0) (6, -0.5) (4,-0.5)};
    \end{tikzpicture}
    }
    \end{subfigure}
    \hspace{18mm}
    \begin{subfigure}[t]{0.4\textwidth}
    \centering
    \resizebox{\textwidth}{!}{%
    \begin{tikzpicture}
    \node[vertex] (v1) at (0,0) {};
    \node[vertex] (v2) at (2,-1) {};
    \node[vertex] (v3) at (2,1) {};
    \node[vertex] (v4) at (4,0) {};
    \node[vertex] (v5) at (6,0) {};

    \draw[line width=0.05cm] plot [smooth cycle] coordinates {(-0.5,0) (2.35,1.35) (2.35, -1.35)};
    \draw[line width = 0.05cm] plot [smooth cycle] coordinates{(3.5, 0) (4,0.5) (6, 0.5) (6.5, 0) (6, -0.5) (4,-0.5)};
    \end{tikzpicture}
    }
    \end{subfigure}%

    \caption{An example of a shape of a non-minimal map (left) and of a minimal map (right). Note that on the left-hand side, the support of one of the edges is completely contained in the combined support of the other edges.}
    \label{fig:my_label}
\end{figure}
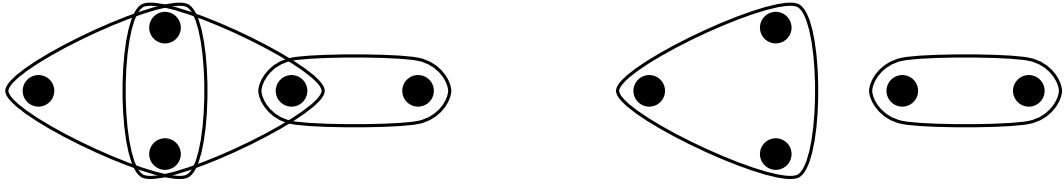
\section{Large intersection sizes}
\label{sec:large}
 {Melo and Winter proved the following proposition, which shows the containment relation $H^+(\infty,k)\supseteq \{2^{k-1}+2^i:i\in \{0,1,\dots,k-1\}\}$.
\begin{proposition}[Proposition 2.7 in \cite{MeloWinter}]
\label{prop:MeloWinter}
Let $k-j \geq t_j>\dots >t_1>t_0\geq 0$. Then $\sum_{i=0}^j 2^{t_i} \in H(\infty, k)$.
\end{proposition}
This reduces Conjecture $A$ to the claim that these are the only values.}
We first investigate the structure of $H^+(k+1,k)$. We will use the observation of Melo and Winter \cite{MeloWinter} that $L(\{e_1,\dots,e_k\})\subseteq \{-1,0,1\}^m$ if $t(L) > 2^{k-1}$. We include their argument here for completeness.

Suppose $\alpha= L_1(e_1)\not\in \{-1,0,1\}$. For any $x\in \{0\}\times \{0,1\}^{k-1}$, if $L_1(x)\in \{0,1\}$, then $L_1(x+e_1)=L_1(x)+\alpha\not\in \{0,1\}$. Hence either $x$ or $x+e_1$ is not in $I(L_1)$, thus $t(L)\leq t(L_1)\leq 2^{k-1}$.
\begin{lemma}
\label{lem:largecodim1}
Suppose $L: \mathbb{R}^k \to \mathbb{R}$ is linear with $L(e_i) \in \{-1, 0, 1\}$ for all $i\in [k]$ and define
\[   a = \left| \left\{ i : L(e_i) = 1 \right\} \right|, \quad b = \left| \left\{ i : L(e_i) = -1 \right\} \right|,  \quad c = \left| \left\{ i : L(e_i) = 0 \right\}\right| .\]
Then
\[ \left| \left\{ x \in \mathbb{H}^k : L(x) = j \right\} \right| = 2^c\binom{a + b}{b + j} .\]
\end{lemma}
\begin{proof}
Since $L(x)=x_1L(e_1)+\dots+x_k L(e_k)$, we find 
\begin{align*}
    \left| \left\{ x \in \mathbb{H}^k : L(x) = j \right\} \right| &= \sum_{i=0}^a 2^c \binom{a}{i} \binom{b}{i-j} \\
    &= 2^c \sum_{i=0}^{b+j} \binom{a}{i} \binom{b}{b + j -i}\\
    &= 2^c \binom{a+b}{b+j} 
\end{align*}
using the Chu-Vandermonde identity in the last line.
Alternatively this can be shown with less calculation in the following, self-contained argument.
We may renumber the $e_i$ such that $e_1,\dots,e_b$ are mapped to $-1$ and $e_{b+1},\dots,e_{a+b}$ are mapped to $1$. Any element $x\in \{0,1\}^k$ is defined by its support, $\text{supp}(x):=\{i\in [k]:x_i\neq 0\}$. The smallest value $L(x)$ can take is $-b$, which is achieved by $\text{supp}(x)=\{1,\dots,b\}$. If an element $i>b$ is added to $\text{supp}(x)$ or an element $i\in [b]$ is removed from $\text{supp}(x)$, then the value of $L(x)$ will increase by 1. In order to get to $0$, $b$ such changes have to be made which can be done in exactly $\binom{a+b}{b}$ ways. The coordinates mapping to 0 (the $x_i$ for $i \geq a+b+1$) have no effect on the value of $L(x)$ so each can individually be included or not; this gives the factor $2^c$. Similarly, the value $j$ is achieved by making $b+j$ switches which can be done in $\binom{a+b}{b+j}$ ways, giving  $2^c\binom{a+b}{b+j}$ possible $x\in \{0,1\}^k$ with $L(x)=j$.
\end{proof}
\begin{corollary}
\label{cor:codim1sizes}
For $k \geq 6$,
    \[ H^+(k, k +1) = \{ 2^{k-1} + 2^{k-5} + 2^{k-6}, 2^{k-1} + 2^{k-3}, 2^{k-1} + 2^{k-2}, 2^{k}\}. \] 
\end{corollary}
\begin{proof}
Let $L : \mathbb{R}^k \to \mathbb{R}$ be a linear map. As noted above, we may assume that $L(e_i) \in \{ -1, 0, 1\}$ for all $i\in [k]$. Define $a, b$ and $c$ as in the previous lemma. Then
\[ t(L) = \left|\left\{ x \in \mathbb{H}^k : L(x) \in \{0,1\} \right\} \right| = 2^c \left( \binom{a+b}{b} + \binom{a+b}{b + 1} \right) = 2^c \binom{a+b+1}{b+1} .\]

 {Note that $t(L)=2^c t(L')$ where $L':\mathbb{R}^{k-c}\to \mathbb{R}$ is the projection of $L$ onto the support of $L$. We therefore first look for triples $(a,b,0)$ with $a + b = k$ which give an intersection $t \geq 2^{k-1} + 1$ (and then multiply the result by $2^c$).} We first show that for such triples to exist, we must have $a+b \leq 7$ for which we can check all possible options. The largest coefficient is the central binomial coefficient so
\[    t(L)= \binom{a+b+1}{b+1} \leq \binom{k+1}{\left\lfloor (k+1)/2 \right\rfloor}.\]
The ratio $2^{-n}\binom{n}{\lfloor n/2\rfloor}$ is non-increasing in $n$, since
\[    \frac{\binom{2n}{n}}{2^{2n}} \frac{2^{2n+1}}{\binom{2n+1}{n}}= \frac{2(n+1)}{2n+1} > 1 \text{ and }     \frac{\binom{2n-1}{n-1}}{2^{2n-1}} \frac{2^{2n}}{\binom{2n}{n}}= \frac{2n}{2n}= 1.\]
Since $2^{-(n+1)}\binom{n+1}{\lfloor (n+1)/2\rfloor}\leq \frac14$ for $n=8$, it follows that $t(L)\leq \binom{k+1}{\lfloor (k+1)/2\rfloor}\leq 2^{k-1}$ for all $k=a+b\geq 8$.
Checking all possible $a+b\leq 7$ (by hand or computer) gives the possible values for $t\geq 2^{k-1}$, which are given in the table below. 
\begin{center}
\begin{tabular}{|c|c|c|c|c|c|c|c|c|c|}
\hline
     $a$ & 1 & 2 & 2 & 2 & 3 & 3 & 3 & 4 & 4\\
     $b$ & 1 & 0 & 1 & 2 & 1 & 2 & 3 & 2 & 3\\
     \hline
     $t$ & 3 & 3 & 6 & 10 & 10 & 20 & 35 & 35 & 70\\
     \hline
\end{tabular}
\end{center}
 {Hence, if $t(L)\geq 2^{k-1}+1$, then $t(L)$ is of the form $2^c t'$ for some $t'$, a value in the table above.} 
\end{proof}
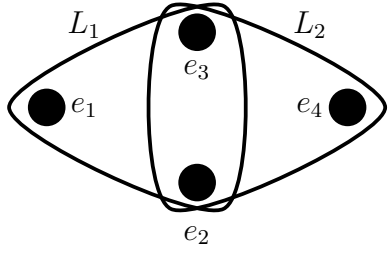
\begin{figure}
    \centering
    \begin{tikzpicture}
    \node[vertex] (v1) at (0,0) {$e_1$};
    \node[]  at (0.5,0) {$e_1$};
    \node[vertex] (v2) at (2,-1) {};
    \node[]  at (2,-1.7) {$e_2$};
    \node[vertex] (v3) at (2,1) {};
    \node[]  at (2,0.5) {$e_3$};
    \node[vertex] (v4) at (4,0) {};
    \node[]  at (3.5,0) {$e_4$};
    \node at (0.5, 1.1) {$L_1$};
    \node at (3.5,1.1) {$L_2$};

    \draw[line width=0.05cm] plot [smooth cycle] coordinates {(-0.5,0) (2.35,1.35) (2.35, -1.35)};
    \draw[line width=0.05cm] plot [smooth cycle] coordinates {(4.5,0) (1.65,1.35) (1.65, -1.35)};
    \end{tikzpicture}
    \caption{A $(3,2)$-star with 2 edges.}
    \label{3-2-star}
\end{figure}
Let $L:\{0,1\}^k\to \R$ be a linear map with $t(L)\geq 2^{k-1}$. We can relabel the coordinates of $\{0,1\}^k$ such that $L(e_i) = 0$ if and only if $i > |S(L)|$. From the above proof of Corollary \ref{cor:codim1sizes} we see that if $|S(L)| = 3$, the non-zero elements $L(e_1), L(e_2)$ and $L(e_3)$ must be $1$, $1$ and $-1$ in some order. This means the elements in $\{0,1\}^3$ that map to $\{0,1\}$ are (up to permuting the coordinates) $\{000, 010, 011, 100, 101, 111\}$.  In particular, if $L(x)  \in \{ 0,1\}$ and $x$ starts $00$, then we know $x$ begins $000$ but if $x$ begins with $01$ then the first three elements may be either $010$ or $011$. We will see below that this is an important part of calculating the intersection size of shape.

Suppose that the shape of $L:\{0,1\}^4\to \R^2$ is a (3,2)-star as in Figure \ref{3-2-star}. To calculate the intersection we condition on the values that the maps $L_1$ and $L_2$ take on the shared basis elements $e_2$ and $e_3$. If they both take the value $1$ on both $e_2$ and $e_3$ and $x = (x_1, x_2, x_3, x_4)$ is such that $L(x) \in \{0,1\}^2$, then given $x_2x_3 = 00$ we know that $x = 0000$. If we are given $x_2x_3 = 01$, then $x_1$ could be 0 or 1 and independently $x_4$ could be 0 or 1 so there are 4 options for $x$. One can continue with this to see that the intersection has size \[ 2 \times 2 + 2 \times 2 + 1 \times 1 + 1 \times 1 = 10.\]

Suppose instead $L_1$ takes the value $1$ on both $e_2$ and $e_3$ but $L_2$ takes the value $-1$ on $e_2$ and 1 on $e_3$ and now attempt to calculate the intersection size. If $x_2x_3 = 00$, we again know that $x_1 = 0$ but now $x_4$ could be 0 or 1 giving 2 options for $x$. If $x_2 x_3 = 01$, then $x_1$ could be 0 or 1 but now $x_4$ must be 0 so we only get 2 options for $x$. Completing the calculation shows that the intersection size is 
\[ 1 \times 2 + 2 \times 1  + 2 \times 1 + 1 \times 2 = 8.\]

Combining the above with a symmetry argument will show that an intersection size strictly greater than half the cube is only possible if the two maps ``agree'' on whether their common intersection has a $-1$, since this gives that the options for the shared coordinates with many extensions, also agree. 

This illustrates the general principle that for large intersection sizes we need the points where each map can extend to a large number of different options to agree to some extent. While this works on all shapes, it quickly becomes very tedious and we will use a computer to calculate the largest intersections of shapes.

Many shapes can be quickly ruled out; for example, if the shape of $L:\{0,1\}^7\to \R^3$ is a ``(3,1)-star with three edges'', say $S(L_1)=\{1,2,3\}$, $S(L_2)=\{1,4,5\}$ and $S(L_3)=\{1,6,7\}$, then counting the number of $x\in I(L)$ with $x_1=0$ and those with $x_1=1$ we find
\[
t(L)=3\times 3 \times 3+3\times 3\times 3 = \left(\frac{3}{4}\right)^32^7<\frac12\times 2^7.
\]
Similar calculations can rule out many other shapes such as having three disjoint edges in your shape.

The main result for large intersections is the following theorem, from which
Theorem \ref{thm:largeinf} follows directly.
\begin{theorem}
\label{thm:largen}
For $k \geq 6$,
\begin{align*}
    H^+(k+1,k) &= \left\{ 2^{k-1} + 2^{k-5} + 2^{k-6},~ 2^{k-1} + 2^{k-3},~ 2^{k-1} + 2^{k-2}, 2^{k} \right\},\\
    H^+(k+2,k) &= H^+(k+1,k) \cup\{ 2^{k-1} + 2^{k-4}\},\\
    H^+(k+i,k) &= H^+(k+i-1,k)\cup \{ 2^{k-1} + 2^{k-(i+1)}\} \text{ for } i = 3, \dots, k-1,\\
   \text{and } H^+(k+i,k) &= H^+(2k-1,k) \text{ for } i \geq k.\\
\end{align*}
\end{theorem}
In particular, note that $H(k+2,k)=H(k+3,k)$ for $k\geq 4$. We also note that it is easy to check that Conjecture \ref{conj:large} is true for $k \in \{1,\dots,5\}$ using Lemma \ref{lem:largecodim1} and a computer search.

To prove Theorem \ref{thm:largen} we require two more auxiliary lemmas, the first of which describes the structure of $L$.
\begin{lemma}
\label{lem:structure}
Suppose $L:\R^k\to \R^m$ is a minimal linear map with $|S(L_i)|>1$ for all $i\in [m]$, $|S(L)|=k \geq 8$ and $t(L) > 2^{k-1}$. Then the sets $S(L_i)$ form either a (2,1)-star or a (3,2)-star.
\end{lemma}
\begin{proof}
As seen in the proof of Corollary \ref{cor:codim1sizes}, the size of $S(L_i)$ must be in $\{2, \dots, 7\}$ for all $i$ or $t(L) \leq t(L_i) \leq 2^{k-1}$. As we have assumed that $|S(L)| \geq 8$, we must have $m \geq 2$. Checking the possible minimal shapes (as defined in Definition \ref{def:minimal} and \ref{def:shape}) the maps $L_1$ and $L_2$ can form the following options (up to switching $L_1$ and $L_2$).
\[ \begin{tabular}{|c|c|c|}
\hline
    $|S(L_1)|$ & $|S(L_2)|$ & $|S(L_1) \cap S(L_2)|$\\
    \hline
    2 & 2& 0\\
    2 & 2 & 1\\
    2 & 3 & 0\\
    2 & 3 & 1\\
    3 & 3 & 0\\
    3 & 3 & 1\\
    3 & 3 & 2\\
    \hline
\end{tabular}\]
In all of these options $|S(L_{\{1,2\}})| \leq 6$ so in fact $m \geq 3$. This can now be repeated for $L_1, L_2$ and $L_3$ and a computer search shows the sets $S(L_i)$ for $i=1,2,3$ form either a (2,1)-star or a (3,2)-star if $t(L) > 2^{k-1}$.
By the symmetry of $\mathbb{H}^m$, this holds for all distinct $i_1, i_2, i_3$ and the result follows.
\end{proof}
\begin{lemma}
\label{lem:double}
Let $L:\R^k\to \R^m$ be a linear map with $k \geq 8$ and  $t(L)\geq 2^{k-1}+1$. Let $X_i = \{ x \in \mathbb{H}^k : x_i = 0\}$. Then one of the following statements holds:
\begin{itemize}
    \item[(a)] $m \geq k -1$ and $t(L) = 2^{k-1} + 1$;
    \item[(b)] there exists an $i\in[m]$ such that $\left|L^{-1} \left( \mathbb{H}^m \right) \cap X_i \right| = \frac{1}{2}t(L)$.
\end{itemize}  
\end{lemma}
\begin{proof}
If $\left|S(L)\right| < k$, then there is some coordinate $i$ such that $L(e_i) = 0$. This means the value $L(x)$ does not depend on whether $x_i$ is 0 or 1, so that $i$ satisfies (b). Thus we may assume that $|S(L)|=k$. Choose $A\subseteq [m]$ such that $L'=L_A$ is minimal and it satisfies $\left|S(L')\right| = k$. All assumptions of Lemma \ref{lem:structure} are then satisfied and so the sets $S(L'_i)$ form either a $(2,1)$-star or a $(3,2)$-star. 

Suppose the $S(L'_i)$ form a (2,1)-star. Then, as $|S(L')| = k$, $|A| = k -1$ and $m \geq k -1$. The maximum intersection size of a (2,1)-star is $2^{k-1} + 1$ so, since we have assumed $t(L)\geq 2^{k+1}+1$, (a) is satisfied.

Suppose the $S(L'_i)$ form a $(3,2)$-star. It is easy to see that $t(L') = 2^{k-1} + 2$ and that (b) holds for $L'$ for every $i$. Either $L$ and $L'$ have the same intersection pattern (that is, $I(L) = I(L')$) and we are done, or $m \geq |A| + 1 \geq k - 1$ and $t(L)\leq t(L')-1\leq 2^{k-1}+1$ so (a) is satisfied. In either case at least one of (a) or (b) is satisfied for $L$.
\end{proof}
\begin{proof}[Proof of Theorem \ref{thm:largen}]
Using the above lemmas we prove the theorem using induction on $k$. 
A computer search shows that the claim holds for $k = 6,7,8$, so assume the theorem holds for some $k \geq 8$.

For any linear map $L: \mathbb{R}^k \to \mathbb{R}^m$, we can define a map $L' : \mathbb{R}^{k+1} \to \mathbb{R}^m$ by $L'(e_i) = L(e_i)$ for $i = 1, \dots, k$ and $L'(e_{k+1}) = 0$. Then $t(L') = 2 t(L)$, so that $H(k+i,k+1)$ contains the claimed elements for $i\in \{1,\dots,k-1\}$. 
The map $L:\R^{k+1}\to \R^k$ defined by 
\[
L_j(e_i)=\begin{cases}
1 & \text{if }i=j\in \{1,\dots,k\} \text{ or }i=k+1\\
0 & \text{otherwise}
\end{cases}
\]
satisfies $t(L)=2^{k}+ 1$, which implies that $H^+(n,k+1)$ contains the claimed elements for all $n \geq k$. 

It remains to show that if $L: \mathbb{R}^{k+1} \to \mathbb{R}^m$ is a linear map that then $t(L)$ is in the claimed set. Lemma \ref{lem:double} shows that either $t(L) \leq 2^{k}+1$ and $m \geq k$  in which case we are done, or there is an $i$ such that $\left|L^{-1} \left( \mathbb{H}^{m} \right) \cap X_i \right| = \frac{1}{2}t(L)$. In the latter case, the map $L':\R^k \to \mathbb{R}^m$ with
\[
L'(e_j)
 = 
\begin{cases} L(e_j) & \text{for $j = 1, \dots, i-1$}\\
L(e_{j+1}) & \text{for $j=i,\dots,k$}\\ \end{cases}
\]
has $|(L')^{-1}(\H^m) \cap \H^k|=|L^{-1}(\H^m) \cap X_i|=\frac{1}{2}t(L) \geq 2^{k-1} + 1 $. It follows that $\frac{1}{2} t(L) \in H^+\left(n-1,k  \right)$ so $t$ must be of the claimed form by the induction hypothesis.
\end{proof}

\section{Small intersection sizes}
\label{sec:small}
To determine the large intersection sizes we made extensive use of the fact that $L$ must map each of the basis vectors into $\{-1,0,1\}^m$. This is not true in general for small intersection sizes but we show that it is true on the interval $\left( (15/16) \cdot 2^{k-1}, 2^{k-1} \right)$. This allows us to use a structural approach similar to the large case.

To see this suppose $L(\{e_1, \dots, e_k\}) \not \subseteq \{-1,0,1\}^m$. Without loss of generality we can assume that, for some $\ell \geq 1$, $L_1(e_i) \neq 0$ for $i=1, \dots, \ell$, $\alpha = L_1(e_{\ell}) \not \in \{-1,0,1\}$ and that $L_1(e_i)  = 0$ for $i \geq \ell  + 1$. Call an element $y \in \{0,1\}^{\ell-1}$ \textbf{good} if either $L_1(y+e_{\ell}) \in \{0,1\}$ or $L_1(y) \in \{0,1\}$. As $\alpha \not \in \{-1,0,1\}$, it can't be the case that both $L_1(y+e_{\ell})$ and $L_1(y)$ are in $\{0,1\}$ and hence \[ t(L_1) = \# \text{good elements} \cdot 2^{k-\ell}.\]
An element $y \in \{0,1\}^{\ell-1}$ is good if and only if $L_1(y) \in \{0,1, 1- \alpha, - \alpha\}$ so we can bound the maximum intersection size using a Littlewood-Offord style result \cite{ErdosOfford,LittlewoodOfford}. 
\begin{lemma}
\label{lem:antichain}
Let $a_1,\dots,a_\ell \in \R\setminus \{0\}$ and $B\subseteq \R$ of size $|B|=4$. Then 
\[
\left|\left\{I\subseteq [\ell ]:\sum_{i\in I} a_i\in B\right\}\right| \leq \frac{15}{16}2^\ell
\]
for all $\ell \geq 4$.
\end{lemma}
\begin{proof}
We associate the set $\{I\subseteq [\ell ]:\sum_{i\in I} a_i\in B\}$ with the (disjoint) union of four antichains in $\mathcal{P}([\ell])$.  
The result then follows from the fact that the union of four antichains contains at most $\sum_{i\in \{-1,0,1,2\}}\binom{\ell}{\lfloor \ell/2+i\rfloor }$ elements, which is at most $(15/16)2^\ell$ for $\ell\geq 4$. This can be verified by checking that $2^{-\ell} \sum_{i\in \{-1,0,1,2\}}\binom{\ell}{\lfloor \ell/2+i\rfloor }$ is non-increasing in $\ell$ and then computing $\ell=4$.

For each $b\in B$, we associate a separate antichain. Suppose first that $a_i>0$ for all $i\in [\ell]$. If $\sum_{i\in I_1}a_i=\sum_{i\in I_2}a_i=b$ with $I_1\subseteq I_2$, then $\sum_{i\in I_2\setminus I_1}a_i=0$ and since $a_i>0$, it follows that $I_1=I_2$. Hence the sets $\{I\subseteq[\ell]:\sum_{i\in I}a_i=b\}$ form an antichain of $\mathcal{P}([\ell])$ for each $b\in B$. 

If $N=\{i\in [\ell]:a_i<0\}\neq \emptyset$, then we can consider $b'=b-\sum_{i\in N}a_i$ and $a_i'=|a_i|$ instead. If $\sum_{i\in I}a_i=b$, then \[
\sum_{i\in (I\setminus N)\cup (N\setminus I)} a'_i=\sum_{i\in I\setminus N}a_i-\sum_{i\in N\setminus I}a_i=\sum_{i\in I}a_i-\sum_{i\in N}a_i=b'.
\]
Now note that $I\mapsto (I\setminus N)\cup (N\setminus I)$ defines a bijection on $\mathcal{P}([\ell])$.
\end{proof}

This shows that the intersection size is small if many $e_i$ are mapped to non-zero elements, something we have seen before in Corollary \ref{cor:codim1sizes}. In fact, the above argument can be adapted to show that any map with at least two non-zero entries can have at most $3/4$ of the hypercube mapping to any set of two elements. 

The following result shows if there are a small number of non-zero elements, adding (non-redundant) constraints must reduce the intersection by a large amount which will cause a gap in the values that can be achieved.

\begin{lemma}
\label{lem:small}
Let $L : \mathbb{R}^k \to \mathbb{R}^{m+1}$ be a linear map such that the last condition is not redundant i.e. $t(L) < t(L_{[m]})$. Then 
\[ t(L) \leq \max \left\{ \tfrac{3}{4} t\left(L_{[m]}\right), t\left(L_{[m]}\right) - 2^{k-s-1} \right\} \]
where $s = |S(L_{[m]})|$.
\end{lemma}
\begin{proof}
Permuting the coordinates as necessary we can write $I(L_{[m]}) = J(L_{[m]}) \times \{0,1\}^{k-s}$. 
Let $L'$ be the restriction of $L_{m+1}$ to the last $k-s$ coordinates. 
\begin{itemize}
    \item Suppose $L'$ maps at least two basis vectors to non-zero values. Then, as observed above, $L'$ can map at most $3/4$ of the hypercube $\{0,1\}^{k-s}$ to any set of two elements. But for any $(x,z)\in I(L_{[m]})$, we find that $(x,z) \in I(L)$ if and only if $L'(z) \in \{-L(x), 1-L(x)\}$. This implies that $t(L) \leq (3/4)t(L_{[m]})$.
    \item Suppose $L'$ maps one basis vector to a non-zero value. Since the last condition is not redundant, there exist $x \in J(L_{[m]})$ and $z \in \{0,1\}^{k-s}$ such that $L(x,z) \not \in \{0,1\}$. 
Every vector $y \in \{0,1\}^{k-s}$ that agrees with $z$ in position $i$ (that is, $y_i=z_i$) satisfies $L_{m+1}(x,y) = L_{m+1}(x,z) \not \in \{0,1\}$. Since there are $2^{k-s-1}$ such vectors, it follows that $t(L) \leq t(L_{[m]}) - 2^{k-s - 1}$.
    \item The case where $L'$ maps everything to zero is similar to the second case (but with the slightly stronger bound $t(L) \leq t(L_{[m]}) - 2^{k-s}$).
\end{itemize}
In all three cases, the statement holds.
\end{proof}



\begin{lemma}
\label{lem:ints}
If $t(L) > \frac{15}{16} 2^{k-1}$ and $t(L) \neq  2^{k-1}$, then $L(\{ e_1, \dots, e_k\}) \subseteq \{-1,0,1\}^m$
\end{lemma}
\begin{proof}
Suppose $L(\{e_1, \dots, e_k\}) \not \subseteq \{-1,0,1\}^m$ and as before assume that, for some $\ell \geq 1$, $L_1(e_1), \dots, L_1(e_{\ell-1}) \neq 0$, $\alpha = L_1(e_{\ell}) \not \in \{-1,0,1\}$ and that $L_1(e_{\ell+1}) = L_1(e_{\ell+2}) = \dotsb = L_1(e_k) = 0$.

By Lemma \ref{lem:antichain}, the set of good elements has size at most $(15/16)2^{\ell-1}$ for $\ell-1\geq 4$. Hence, if $\ell\geq 5$,
\[ t(L) \leq t(L_1) \leq \frac{15}{16}2^{\ell-1} 2^{k-\ell} = \frac{15}{16}2^{k-1}. \]
Suppose $\ell \leq 4$, then either $t(L_1) = 2^{k-1}$ or \[t(L) \leq t(L_1) \leq \frac{7}{8}2^{k-1} \leq  \frac{15}{16} 2^{k-1}.\]
If $t(L_1) = 2^{k-1}$, then either $t(L) = t(L_1)= 2^{k-1}$ or, by Lemma \ref{lem:small}, \[t(L) \leq \max \left\{ \frac{3}{4} 2^{k-1}, \left(1 - 2^{-\ell} \right)2^{k-1} \right\} \leq \frac{15}{16} 2^{k-1} . \]
\end{proof}

Armed with the lemmas above we are now ready to prove the main result of this section: for $k \geq 8$ and any $m \geq 1$,
\[ H(k+m, k) \cap \left\{ \tfrac{15}{16} 2^{k-1}, \dots, 2^{k-1}\right\}= \left\{ \tfrac{15}{16} 2^{k-1}, \tfrac{126}{128} 2^{k-1}, 2^{k-1}  \right\}.\]
\begin{proof}[Proof of Theorem \ref{thm:smalln}]
Lemma \ref{lem:largecodim1} shows that the claimed values are present for $m =1$ and hence for all $m \geq 1$. This reduces the problem to showing that no other values are present.

By Lemma \ref{lem:ints}, we may assume that $L$ maps the basis vectors $e_i$ to $\{-1, 0, 1\}^m$ and, without loss of generality, we can also assume that no conditions are redundant i.e. there is no $i \in [m]$ such that $t(L) = t(L_{[m]\setminus \{i\}})$. If $|S(L)| < k$, then the result for $L$ follows from the result for $L_{S(L)}$ the restriction of $L$ to $S(L)$. Hence we only need to prove the result for the maps $L$ such that $|S(L)| = k$.

Having made these assumptions we now proceed as in the large case and look for the shapes which can have a sufficiently large intersection size. We must have $|S(L_i)| \in \{2, \dots, 9\}$ for all $i$ else a condition is redundant or $t(L) \leq t(L_i) \leq (15/16)2^{k-1}$. A computer search shows that there are 18 options for the shape formed by $S(L_1)$ and $S(L_2)$ which have a sufficiently large intersection size but all of them have $|S(L_{1,2})| \leq 6$. As we have assumed $|S(L)| = k \geq 8$, we must have $m \geq 3$ and a similar argument shows that $m \geq 4$. The search shows that either $t(L) \leq t(L_{[4]}) \leq (15/16)2^{k-1}$ or the sets $S(L_1), \dots, S(L_4)$ form one of six different shapes. These shapes are 
\begin{itemize}
     \item the (2,1)-star with 4 edges,
     \item the (3,2)-star with 4 edges,
     \item the (2,1)-star with 3 edges plus a second copy of one of the edges,
     \item the (3,2)-star with 3 edges plus a second copy of one of the edges,
     \item the (3,2)-star with 3 edges where there is an edge containing just the two central vertices,
     \item  and the (3,2) with 2 edges where there are 2 edges containing just the two central vertices.
\end{itemize}

Every subset of 4 edges must form one of the above shapes and so it follows that the only shapes with $m \geq 4$ and sufficiently large intersection size are extensions of the above formed by adding extra edges to the star. By conditioning on the value over the centre of the star, we find that the largest intersection size smaller than $2^{k-1}$ is bounded above by $2^{k-2} + 4$ which, for $k \geq 8$, is less than $(15/16)2^{k-1}$.
\end{proof}
\section{Computer search}
This section gives a brief overview of the computer searches that were instrumental in proving the preceding theorems. 
The aim of each search is to identify the possible shapes $\{S_i:i\in [m]\}$ for which there is a corresponding map with a large intersection size. This is done by starting with the shapes for $m=1$ with sufficiently large intersection (these are easily found by checking binomial coefficients) and incrementally adding conditions (adding new edges $S_{2}, S_{3}, \dots$). 

As adding conditions does not increase the proportion of the hypercube that is covered, only shapes which cover a sufficiently high proportion need to be kept. Isolated points in the hypergraph do not affect the proportion of points covered by a shape so are ignored by the computation. We do a breadth-first search over the possible shapes where, at each step, every way of adding another condition is checked. It is clearly possible to add the conditions in any order and we make use of this by adding conditions in order of non-increasing support size. This reduces the computation required and removes some isomorphic shapes although we make no other attempt to eliminate isomorphic shapes and these must be removed by hand later. In particular, the code for the proof of Theorem \ref{thm:largen} will output ten shapes but only six of them are different.

A shape does not uniquely define a map and different maps with different intersection sizes can have the same shape. For each shape that is considered, the maximum intersection size is computed by checking all assignments of $\pm 1$ to every appropriate $L_i(e_j)$ (where $L_i(e_j)\neq 0$ if and only if the shape dictates that $j\in S(L_i)$). Suppose that the coordinates that are in multiple edges are $1, \dots, \ell$ and the coordinates $\ell+1, \dots, k$ are in exactly one edge. We consider every point $x \in \{0,1\}^\ell$ and, for each edge $L_i$, we calculate $L_i(x)$. Given this value, we can use Lemma \ref{lem:largecodim1} to calculate the number of ways of extending $x$ to the other coordinates in $S(L_i)$. This is done for every map and the total number of ways of extending $x$ to $[k]$ is then the product of the number of ways of extending each map. For a given map, the number of ways of extending $x$ depends only on the number of $1$s (and the number of $-1$s) and not which coordinates they correspond to. This is used to reduce the computation and speed up the search.

Removing any set $S_i$ from a minimal shape leaves another minimal shape (though for a smaller domain size $k$). This means that we will consider all minimal shapes by starting from a minimal shape and attempting to add all possible edges that keep the shape minimal.

For the last proof we cannot restrict to minimal shapes but we instead require that there are no redundant conditions (i.e. there is no $i$ such that $t(L)=t(L_{[m]\setminus \{i\}})$). For simplicity and to keep the search quick, the search does the following which may allow some assignments with redundant conditions. Suppose we are attempting to add a set $S_{m+1}$ to the shape with sets $S_1, \dots, S_m$ which has some previously computed maximum intersection size $t$. When calculating the maximum intersection size of the new shape, any assignments that achieve an intersection of size at least $t$ are ignored as these must have at least one redundant condition. (Note that it is possible that the redundant condition is not the last one.) The shape is kept if the largest intersection size, among those which were not thrown away for being redundant, is sufficiently large. 

The code can be found attached to the arXiv submission.

\section{Conclusion}
While we have shown that Conjecture \ref{conj:small} is false, there is still much that is not known about the small intersections.  {It seems likely that, as the size of the cubes increases to infinity, there are values missing in the lower $2^{-M}$ fraction for any fixed $M$. We in fact make the following stronger conjecture that there is a positive proportion of such values missing.
\begin{conjecture}
For every fixed $M\in \N$, there exist $\epsilon>0$ and $k_0\in \N$ such that $|[0,2^{k-M}]\cap H(\infty,k)|\leq (1-\epsilon)2^{k-M}$ for all $k\geq k_0$.
\end{conjecture}
If this conjecture is true, then another interesting question is whether a positive proportion of the values is missing in $[\alpha 2^k,\beta2^k]$ for any fixed $0<\alpha<\beta<1$ and all $k$ sufficiently large (depending on $\alpha$ and $\beta$).}

From our results, one can also draw conclusions about the structure of the sets $H(n)=\bigcup_{k\leq n}H(n,k)$, that is, the possible number of points we can intersect a fixed cube with using linear subspaces. For example, by Corollary \ref{cor:codim1sizes},
\begin{align*}
H(n)\cap \left[\frac14 2^{n},2^n\right]&=
\left(H(n,n)\cup H(n,n-1) \cup H(n,n-2)\right) \cap \left[\frac14 2^{n},2^n\right]\\
&= \left\{2^n,2^{n-2} + 2^{n-6} + 2^{n-7}, 2^{n-2} + 2^{n-4},2^{n-2} + 2^{n-3}, 2^{n-1},2^{n-2}\right\}\\
&= \left\{\frac1{4}2^n, \frac{35}{128}2^n, \frac{5}{16}2^{n}, \frac38 2^{n}, \frac12 2^{n},2^n\right\}.
\end{align*}

\paragraph{Acknowledgements.} We would like to thank James Aaronson for useful discussions and Alex Scott for suggesting the problem.

\bibliographystyle{plain}
\bibliography{intersections}

\begin{thebibliography}{1}

\bibitem{ErdosOfford}
Paul Erd\hack{o}s and A.~C. Offord.
\newblock On the number of real roots of a random algebraic equation.
\newblock {\em Proceedings of the London Mathematical Society},
  s3-6(1):139--160, 1956.

\bibitem{LittlewoodOfford}
J.~E. Littlewood and A.~C. Offord.
\newblock On the number of real roots of a random algebraic equation.
\newblock {\em Journal of the London Mathematical Society}, s1-13(4):288--295,
  1938.

\bibitem{MeloWinter}
N.~Melo and A.~Winter.
\newblock Intersection patterns of linear subspaces with the hypercube, 2017.
\newblock arXiv:1712.01763.

\end{thebibliography}
\end{document}